%% file: PaperEndoAuto.tex
\documentclass[a4paper, 11pt]{amsart}
\input{preamble}

\newcommand{\C}{\mathbb{C}}

\newcommand{\Q}{\mathbb{Q}}
\newcommand{\R}{\mathbb{R}}
\newcommand{\Z}{\mathbb{Z}}

\newcommand{\id}{\operatorname{id}}

\newcommand{\fin}{\operatorname{Fin}}

\newcommand{\SO}{\mathrm{SO}}
\newcommand{\LO}{\mathrm{O}}

\newcommand{\SC}{\mathcal{SC}}
\newcommand{\Aut}{\mathrm{Aut}}

\newcommand{\map}{\operatorname{Map}}
\newcommand{\mul}{\operatorname{Mul}}
\newcommand{\mco}{\mathcal O}
\newcommand{\Cat}{\operatorname{Cat}}
\newcommand{\GL}{\text{GL}}
\newcommand{\spc}{\operatorname{Spc}}

\usepackage{dynkin-diagrams}

\usepackage[backend=bibtex, style=alphabetic]{biblatex}
\bibliography{ref.bib}

\title[Endomorphisms and Automorphisms of Framed Little Disk Operad]{Every Endomorphism of the Framed Little Disk Operad is an Automorphism}
\author{Alice Rolf}
\address{University of Toronto}
\email{alice.rolf@mail.utoronto.ca}

\date{\today}

\begin{document}
\begin{abstract}
In a recent paper, Horel–Krannich–Kupers proved that all endomorphisms of the little $d$-disk operad are automorphisms. In this paper we show that this is also true for the framed little $d$-disk operad by using the classification of self maps of simple Lie groups. We also examine whether this property holds for the swiss cheese operad and prove that it holds for some other semidirect products of a group with a little disk operad.
\end{abstract}
\maketitle
Let $G$ be a group acting on the little disk operad $E_d$ in the sense that $G$ acts on $E_d(n)$ and the operadic structure maps are $G$-equivariant. Then we can form the semidirect product $E_d^G$ of $G$ with $E_d$ \cite{salvatore2003framed}. The stereotypical example of such an operad is the framed little disks operad $\smash{E_d}^{\SO(d)}$, first introduced by Getzler in the context of topological field theories \cite{getzler1994batalin}. More recently, the study of automorphisms of the (framed) little disks operad $\Aut(E_d^G) = \Aut_{\op}(E_d^G)$ has come up in several parts of geometric topology, including embedding calculus \cite{dwyer2012long, boavida2018spaces, ducoulombier2022delooping}, and the Grothendieck Teichmüller group \cite{willwacher2015m}.
Here we should emphasize that when we talk about mapping spaces (resp. automorphisms) of operads we mean the derived mapping spaces (resp. derived automorphisms) of the associated $\infty$-operads.\\
In \cite{horel2022two} the authors show the surprising result that all derived endomorphisms of the little disk operad are automorphisms. We prove that this is also true for the framed little disk operad:
\begin{atheorem}\label{maintheorem}
	All endomorphisms of $\smash{E_d}^{\SO(d)}$ are automorphisms.
\end{atheorem}
Using this theorem and its proof we deduce the following theorem in the last section of this paper:
\begin{atheorem} Let $d>0$.
\begin{enumerate}
	\item All endomorphisms of $\smash{E_d}^{\text{O}(d)}$ are automorphisms.
	\item All endomorphisms of $\smash{E_{2d+1}}^{\SO(2d)}$ are automorphisms.
	\item All color-preserving endomorphisms of the (framed) swiss cheese operad $\SC_d$ are automorphisms.
\end{enumerate}
\end{atheorem}

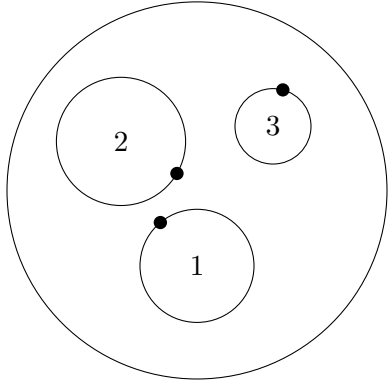
\begin{figure}[t]
\centering 
\begin{tikzpicture}[scale=0.5]
\draw (0,0) circle[radius = 5];
\draw (0,-2) node{1} circle[radius = 1.5] ;
\fill[black] (0,-2) ++(130:1.5) circle[radius=5pt];
\draw (2,1.7) node{3} circle[radius = 1] ;
\fill[black] (2,1.7) ++(75:1) circle[radius=5pt];
\draw (-2,1.3) node{2} circle[radius = 1.7] ;
\fill[black] (-2,1.3) ++(-30:1.7) circle[radius=5pt];
\end{tikzpicture}
\caption{An element in $E_2^{SO(2)}(3)$}
\end{figure}

In order to prove both theorems we are going to show the following result in Section 1 (see \Cref{lemma:automorphismcorrespondence}): The space of all endomorphisms of $E_d^G$ is exactly the space of all maps $f \colon BG \rightarrow BG$ making the diagram
\begin{equation}\label{diagramoperad}
\begin{tikzcd}
	BG \ar[r, "f"] \ar[d] & BG\ar[d]\\
	B\text{Aut}(E_d) \ar[r,"\id"] & B\text{Aut}(E_d)
\end{tikzcd}
\end{equation}
commute up to specified homotopy. Here, the vertical maps are induced by the action of $G$ on $E_d$.
In particular, showing that all endomorphisms of $E_d^{G}$ are equivalences is the same as showing that all maps $f \colon BG \rightarrow BG$ over $B\Aut(E_d)$ are invertible.\\
To prove \Cref{maintheorem} we show that all maps $f \colon B\SO(d) \rightarrow B\SO(d)$ making
\[
\begin{tikzcd}
	B\SO(d) \ar[r,"f"] \ar[d] & B\SO(d) \ar[d]\\
	B\Aut(S^{d-1}) \ar[r,"\id"] & B\Aut(S^{d-1})
\end{tikzcd}
\]
commute are homotopic to the identity, and therefore equivalences. This result is stronger than what we need: We only end up looking at the self maps of $B\SO(d)$ over arity 2 of $E_d$ instead of all of $E_d$. Further, any map making \eqref{diagramoperad} commute will automatically make the above diagram commute as well.
\tableofcontents

\subsection*{Acknowledgements}
I want to thank my advisor Sander Kupers for posing the main question to me and for insightful discussions. This work is partially supported by the Velimir Jurdjevic Graduate Scholarship. I thank Bob Oliver, Ismael Sierra del Rio, and Maarten Mol for helpful comments.
\section{Set-Up}
\subsection{Definitions}
In what follows, when we say category we implicitly mean $\infty$-category and when we say groupoids or spaces we mean $\infty$-groupoids. We will denote the $\infty$-category of spaces or $\infty$-groupoids by $\spc$ and the $\infty$-category of categories by $\Cat$. For $\infty$-categories and $\infty$-operads we will use Lurie's models as established in \cite{lurie2009higher} and \cite{Lurie2017higher}.\\
We recall some definitions of Lurie's model of $\infty$-operads from \cite{Lurie2017higher}: Lurie defines an operad $\mco$ as a functor of $\infty$-categories
\[
\mco^\otimes \rightarrow \fin_*
\]
into the category $\fin_*$ of finite pointed sets satisfying certain lifting properties \cite[\nopp 2.1.1.10]{Lurie2017higher}. Morphisms between operads are functors over $\fin_*$ preserving certain lifts \cite[\nopp 2.1.2.7]{Lurie2017higher}, i.e. the category of operads $\op$ is a subcategory of $\Cat_{/\fin_*}$. We refer to
\[
\mco^{\text{col}}:= \operatorname{fib}_{\langle 1\rangle}(\mco^\otimes \rightarrow \fin_*)
\]
as the category of colors where $\langle 1 \rangle = \{*,1\}$. An object in $\mco$ can be identified with a sequence $(c_s)_{s \in S}$ where $c_s \in \mco^\text{col}$ is indexed by a finite set $S$ \cite[\nopp 2.1.1.5]{Lurie2017higher}.\\
Lurie's definition of an operad relates to the classical 1-categorical definition through its space of multioperations \cite[\nopp 2.1.1.16 and 2.1.1.1]{Lurie2017higher}: Let $d \in \mco^\text{col}$ be another color. We define the space of multi-operations $\mul_\mco ((c_s)_{s \in S}; d)\subset \map_\mco((c_s)_{s \in S};d)$ as the components that map in $\fin_*$ to $\varphi \colon S \sqcup \{*\}\rightarrow \langle 1 \rangle$ where $\varphi$ is the unique active map, i.e. $\varphi^{-1}(\ast) = \ast$. Those spaces  of multi-operations admit a composition product that fulfill the axioms of classical 1-categorical operads up to higher coherence \cite[\nopp 2.1.1.17]{Lurie2017higher}.
\begin{definition}
\begin{enumerate}
	\item An operad is called \emph{unital} if $\mul_\mco(\emptyset; c)$ is contractible for all $c \in \mco^\text{col}$. We refer to $\op^\text{un}\subset \op$ as the full subcategory of all unital operads \cite[\nopp 2.3.1.1]{Lurie2017higher}.
	\item If in addition $\mco^\text{col}$ is a groupoid, we call $\mco$ a \emph{groupoid colored} operad. We refer to $\op^\text{gc}\subset \op^\text{un}$ as the full subcategory of all groupoid-colored operads. 
	\item Finally, a groupoid colored operad is called \emph{reduced} if $\mul_\mco(c; d)$ is contractible for all colors $c,d \in \mco^\text{col}$. We refer to $\op^\text{red}\subset \op^\text{gc}$ as the full subcategory of all reduced operads.
\end{enumerate}
\end{definition}
In \cite[\nopp 2.2]{krannich2024infty} the authors deduce the following theorem from the theory of assembly and disintegration of operads in \cite{Lurie2017higher}:
\begin{theorem}\label{theorem:theoremtangentialstructure}
	Let $\theta \colon B \rightarrow \op^\text{red}$ be a functor where $B$ is a groupoid. Then, there is an equivalence of categories
\[
(\operatornamewithlimits{colim}_{B}\theta)^\text{col} \simeq B
\] 
Furthermore, sending such a functor $\theta$ to its colimit induces an equivalence
	\[\begin{tikzcd}
		\int_{\spc}\operatorname{Fun}(-, \op^\text{red}) \ar[rr, "\simeq"] \ar[dr] && \op^\text{gc} \ar[dl]\\
		&\spc &
	\end{tikzcd}\]
	in $\operatorname{Cat}_{/\spc}$ where $\int_{\spc} \operatorname{Fun}(-, \op^\text{red})$ denotes the Grothendieck construction.
\end{theorem}
Informally speaking, this theorem states that groupoid-colored operads can be seen as families of reduced operads.\\
Using this, we can define a tangential structure as in \cite[\nopp 2.4]{krannich2024infty}:
\begin{definition}
	Let $\mco$ be a reduced operad and $B$ a space. A \emph{tangential structure} for $\mco$ is a functor $\theta \colon B \rightarrow B\Aut(\mco)$. Then, the $\theta$-framed operad is defined as
	\[
	\mco^\theta := \operatorname{colim}(B \xrightarrow{\theta} B\Aut(\mco) \hookrightarrow \op^\text{red}) \in \op^{gc}.
	\]
\end{definition}
\begin{remark}
The above definition compares to the classical way of describing tangential structures on reduced operads as follows:\\
Let $P$ be an operad in the category of spaces in the classical sense equipped with an action of a group $G$. We can form the semidirect product $P \rtimes G$ as done in \cite[\nopp 2.1]{salvatore2003framed}. The space of $n$-ary operations is given by $P(n) \times G^n$ and the operadic composition is induced by the composition of $P$ and the action of $G$ on $P$.
The operadic nerve, as constructed in \cite[\nopp 2.1.1.27]{Lurie2017higher}, turns an ordinary operad $P$ into an $\infty$-operad $N^\otimes(P)$. If $P$ has contractible 0-ary and 1-ary operations, then $N^\otimes(P)$ is a reduced operad. Further, the action of $G$ induces a map $\theta \colon BG \rightarrow B\Aut(N^\otimes(P))$ \cite[proof of 2.10]{krannich2024infty}. It turns out that $N^\otimes(P)^\theta$ is equivalent to $N^\otimes(P\rtimes G)$ \cite[\nopp 2.10]{krannich2024infty}. 
\end{remark}
We are now ready to prove:
\begin{lemma}\label{lemma:automorphismcorrespondence}
	Let $G$ be a group acting on $E_d$ through a map 
	\[
	\theta \colon BG \rightarrow B\Aut(E_d)
	\]
	Then, endomorphisms of $E_d^G$ correspond to maps $f \colon BG \rightarrow BG$ making the following diagram commute up to homotopy: 
	\[
\begin{tikzcd}
	BG \ar[r, "f"] \ar[d, "\theta"] & BG \ar[d, "\theta"] \\
	B\Aut(E_d) \ar[r, "\id"]& B\Aut(E_d)\mathrlap{.}
\end{tikzcd}
\]
\end{lemma}
\begin{proof}
By \Cref{theorem:theoremtangentialstructure} we can calculate the mapping space in the Grothendieck construction $\int_{\spc}\operatorname{Fun}(-, \op^\text{red})$.\\
By definition, objects in $\int_{\spc}\operatorname{Fun}(-, \op^\text{red})$ are given by pairs $(S,F)$ where $S \in \spc$ is a space and $F\colon S \rightarrow \op^{\text{red}}$ is a functor. Morphisms between $(S,F)$ and $(S',F')$ are given by maps $f \colon S' \rightarrow S$ and natural transformations $f^*(F) \Rightarrow F'$:
\[
\begin{tikzcd}
S' \ar[r, "f"] \ar[dr, swap, "F'"{name = M}] & S \ar[d, "F"] \ar[Rightarrow,dl, shorten <=5pt, shorten >= 5pt, to=M] \\
& \op^{\text{red}} \mathrlap{.}
\end{tikzcd}
\]
By \Cref{theorem:theoremtangentialstructure}, $E_d^G$ corresponds to $(BG, \theta)\in \int_\mathcal S \operatorname{Fun}(-, \op^{\text{red}})$. Therefore, an endomorphism of $E_d^G$ is given by a map $f \colon BG \rightarrow BG$ and a natural transformation making the following diagram of functors commute:
\begin{equation}\label{diagramwithopred}
\begin{tikzcd}
	BG \ar[r, "f"] \ar[d, "\theta"] & BG \ar[d, "\theta"] \\
	\op^\text{red} \ar[r, "\id"]& \op^\text{red}\mathrlap{.}
\end{tikzcd}
\end{equation}
$\theta$ factors through the subcategory of operads that are equivalent to $E_d$. By \cite[Theorem B]{horel2022two} we have that all endomorphisms of $E_d$ are automorphisms. Therefore, the full subcategory of $\op^{\text{red}}$ consisting of operads equivalent to $E_d$ is the groupoid $B\Aut(E_d)$. Thus, the diagram \eqref{diagramwithopred} of categories is equivalent to the following diagram of spaces that commutes up to specified homotopy:
\[
\begin{tikzcd}
	BG \ar[r, "f"] \ar[d, "\theta"] & BG \ar[d, "\theta"] \\
	B\Aut(E_d) \ar[r, "\id"]& B\Aut(E_d)\mathrlap{.}
\end{tikzcd}
\]
\end{proof}

\subsection{Endomorphisms of the $E_2$ operad}
In \cite[Section 3]{horel2022two} the authors show that every endomorphism of $E_d$ induces a homology isomorphism in the homology of $E_d(n)$ for all $n$ and for all $d$. If $d \geq 3$ this is enough to show that every endomorphism is an automorphism as $E_d(n)$ is simply connected for all $n$.\\
In the case $d=2$, \cite[\nopp 8.5]{horel2017profinite} calculates the endomorphisms of $E_2$ which shows that all endomorphisms of $E_2$ are automorphisms. We will show this result independently of \cite{horel2017profinite} using less complex tools:
\begin{proposition}
	If an endomorphisms of $E_2$ induces a homology isomorphism in each arity, then it is an automorphism.
\end{proposition}
\begin{proof}
Note that $E_2(n)\simeq K(PBr_n,1)$ where $PBr_n$ denotes the pure braid group on $n$ strands \cite[\nopp 5.1.11]{fresse2017homotopy}. Let $f \colon E_2 \rightarrow E_2$ be an endomorphism that induces an isomorphism on the homology of each arity of $E_2$. By the main theorem of \cite{stallings1965homology} we have that $f$ induces an isomorphism
\[
PBr_n/\Gamma_\alpha(PBr_n) \rightarrow PBr_n/\Gamma_\alpha(PBr_n)
\]
where $\Gamma_\alpha(PBr_n)$ denotes the intersection of the terms in the lower central series of $PBr_n$. However, as the pure braid groups are residually nilpotent by \cite{falk1988pure} we have that $\Gamma_\alpha(PBr_n)$ is trivial which implies that $f$ induces an isomorphism of the pure braid groups. Therefore, $f$ is an automorphism.
\end{proof}

\section{Compact Lie groups and maps on their classifying spaces}
\subsection{Compact Lie groups}
In what follows, we will collect some facts we need on Lie groups and their automorphisms. All Lie groups are assumed to be compact and connected.
\begin{definition}
\begin{enumerate}
	\item A \emph{simple} Lie group $G$ is a non-abelian Lie group without connected normal non-trivial subgroups. Equivalently, they are exactly those Lie groups who have a simple Lie algebra $L(G)$, that is $L(G)$ is a non-abelian Lie algebra with no non-trivial ideals \cite[114]{stillwell2008naive}.
	\item A Lie group is \emph{semi-simple} if its Lie algebra is a direct sum of simple Lie algebras \cite[105]{knapp1996lie}.
\end{enumerate}
\end{definition}
\begin{example}\label{example:simplesemisimple}
\begin{enumerate}
\item For $n\neq 2,4$, the Lie algebra $\mathfrak{so}(n)$ is a simple Lie algebra, that is $\SO(n)$ is simple \cite[33,130]{stillwell2008naive}.
\item $\SO(2)$ is not simple as it is an abelian group.
\item $\SO(4)$ is a semisimple Lie group \cite[75]{stillwell2008naive}.
\end{enumerate}
\end{example}

\begin{definition}
	A \emph{torus} is a Lie group such that $T \cong U(1)^n \cong (S^1)^n$ for some $n$.\\
	A \emph{maximal torus} in $G$ is a subgroup $T \subset G$ such that $T$ is a torus and for any other torus $T'$ with $T \subset T' \subset G$ we have $T = T'$.
\end{definition}
\begin{theorem}[{\cite[\nopp 4.23]{adams1982lectures}}]
	Any two maximal tori in $G$ are conjugate to each other. Therefore, they all have the same dimension.
\end{theorem}
\begin{example}\label{exampletorus}
A maximal torus in $\SO(n)$ can be described as follows:
\begin{enumerate}
	\item In $\SO(2n)$, a maximal torus is given by:
		\[
		T := \left \{ \left. \begin{pmatrix}
			D_1 &  & \\
			 & \ddots & \\
			 & & D_n
		\end{pmatrix} ~ \right\vert D_j = \begin{pmatrix}
			\cos(2 \pi i x_j) & -\sin(2 \pi i x_j)\\
			\sin(2 \pi i x_j) & \cos(2\pi i x_j)
		\end{pmatrix} \text { and } x_j \in \R \right \}.
		\]
		Therefore, $T \cong (S^1)^n$
	\item In $\SO(2n+1)$ a maximal torus is the image of the inclusion $T \subset \SO(2n) \subset \SO(2n+1)$, i.e. we obtain that maximal tori in $\SO(2n+1)$ are $n$-dimensional.
\end{enumerate}
Proofs and details are given in \cite[4.19 and 4.20]{adams1982lectures}
\end{example}

\begin{definition}[{\cite[\nopp 4.11]{adams1982lectures}}]
	Let $T \subset G$ be a maximal torus and $e \in G$ be the unit element. The \emph{lattice} is the subset of $L(T)$ given by $\text{exp}^{-1}(e)$ where $\exp \colon L(T) \rightarrow T$.
\end{definition}

\begin{definition}
	The \emph{Weyl group} $W$ of $G$ is given by $N_G(T)/T$ where $N_G(T)$ denotes the normalizer of $T$ in $G$.\\
Equivalently, the Weyl group is the group of all automorphisms of $T$ which are the restrictions of inner automorphisms of $G$ 	\cite[\nopp 4.29]{adams1982lectures}.
\end{definition}

\begin{example} We describe the Weyl group of $\SO(n)$:
\begin{enumerate}
	\item The Weyl group of $\SO(2n+1)$ is given by the group of permutations $\varphi$ of the set $\{-n, \dots, -1,1,\dots, n\}$ such that $\varphi(-r) = - \varphi(r)$. Equivalently, it is the semidirect product $G(n) = (\Z/2\Z)^n\rtimes \Sigma_n$.
	\item The Weyl group of $\SO(2n)$ is $SG(n)$, which is the subgroup of $G(n)$ of even permutations. In terms of the semidirect products, it is the subgroup of $(\Z/2\Z)^n \rtimes \Sigma_n$ consisting of all elements such that the number of -1's in $(\Z/2\Z)^n$ is even.
\end{enumerate}
Proofs and details can be found in \cite[Theorem 3.6]{brocker2013representations}.
\end{example}

\subsection{Maps on classifying spaces of Lie groups}
We recall some results from \cite{jackowski1992homotopy, jackowski1995self}.
Self maps of the classifying space of simple Lie groups are induced by maps on the underlying Lie group and the unstable Adams operations:
\begin{definition}
	Let $G$ be a compact connected Lie group. An \emph{unstable Adams operation of degree $k$} is a map $\psi^k \colon BG \rightarrow BG$ that induces multiplication by $k^i$ in degree $2i$ of rational cohomology.
\end{definition}
\begin{theorem}
For any compact connected Lie group there is up to homotopy at most one unstable Adams operation of degree $k$ \cite[Theorem 1]{jackowski1992homotopy}.
\end{theorem}
The smash product of monoids with a zero element is given by
\[
M_1 \wedge M_2 = (M_1 \times M_2)/\langle(m_1, 0) \sim (0,0) \sim (0,m_2), m_i \in M_i\rangle
\]
\begin{theorem}[\cite{jackowski1992homotopy}, Theorem 2] \label{theoremonclassifcationofselfmapsbg}
	Let $G$ be a simple compact connected Lie group with Weyl group $W$. Then, there is an isomorphism of monoids with a zero element
	\[
	\operatorname{End}(G)/\operatorname{Inn}(G) \wedge \{k \geq 0 \mid k= 0 \text{ or } \mathrm{gcd}(k, |W|) = 1\} \rightarrow [BG, BG]
	\]
	sending $(\alpha ,k)$ to $\psi^k \circ B\alpha$. Here, $\operatorname{Inn}(G)$ denotes the inner automorphisms of $G$ and $\psi^k$ denotes an unstable Adams operation of degree $k$.
\end{theorem}
\begin{corollary}\label{cor:so2n+1maps}
	Let $n \geq 1$. Then, we have an isomorphism
	\begin{gather*}
	\{\operatorname{triv}, \id_{\SO(2n+1)}\} \wedge \{k \geq 0 \mid k= 0 \text{ or } \mathrm{gcd}(k, 2^n \cdot n!) = 1\} \\
	\rightarrow [B\SO(2n+1), B\SO(2n+1)]
	\end{gather*}
	where $\operatorname{triv}$ denotes the trivial self map of $\SO(2n+1)$.
\end{corollary}
\begin{proof}
By \Cref{example:simplesemisimple}, $\SO(2n+1)$ is a simple Lie group. Therefore, the above theorem applies and we need to calculate the right-hand side in \Cref{theoremonclassifcationofselfmapsbg}.\\
$\operatorname{End}(\SO(2n+1))/\operatorname{Inn}(\SO(2n+1))$ only consists of the identity map and the trivial map: A smooth map $\SO(2n+1) \rightarrow \SO(2n+1)$ representing an element in $\operatorname{End}(\SO(2n+1))/\operatorname{Inn}(\SO(2n+1))$ corresponds exactly to a $2n+1$-dimensional representation of $\SO(2n+1)$. By \cite[Theorem 7.7]{adams1982lectures} there are only two irreducible representations of dimension less than or equal to $2n+1$ which are given by the trivial representation and the identity.\\
Furthermore, the Weyl group of $\SO(2n+1)$ is given by $G(n) = (\Z/2\Z)^n\rtimes \Sigma_n$ and has $2^n \cdot n!$ elements.
\end{proof}

\begin{corollary}
Let $n \geq 3$. Then, we have an isomorphism
\begin{gather*}
	\{\operatorname{triv}, \id_{\SO(2n)}, \operatorname{r} \} \wedge \{k \geq 0 \mid k= 0 \text{ or } \mathrm{gcd}(k, 2^{n-1} \cdot n!) = 1\} 
	\\ \rightarrow [B\SO(2n), B\SO(2n)]
\end{gather*}
where $\operatorname{triv}$ denotes the trivial self map of $\SO(2n)$ and $r$ the conjugation by the diagonal matrix $\operatorname{diag}(-1, 1, \dots, 1)$.
\end{corollary}
\begin{proof}
By \Cref{example:simplesemisimple}, $\SO(2n)$ is a simple Lie group. Therefore, we can again apply \Cref{theoremonclassifcationofselfmapsbg} to calculate $[B\SO(2n),B\SO(2n)]$.\\
We first want to calculate the outer autmorphisms $\operatorname{Out}(\SO(2n))$. By \cite[Proposition D.40]{fulton2013representation}, outer automorphisms of the Lie algebra $\mathfrak{so}(2n)$ corresponds to automorphisms of its Dynkin diagram. For $n\neq 4$ we therefore obtain that $\mathfrak{so}(2n)$ has two outer automorphisms. Both of these come from outer automorphisms of $\SO(2n)$: Identity and conjugation by $\operatorname{diag}(-1,1,\dots, 1)$.\\
For $n = 4$, $\mathfrak{so}(8)$ has 6 outer automorhisms some of which are of order 3. This automorphism corresponds to the triality automorphism $ t \colon \mathfrak{so}(8) \rightarrow \mathfrak{so}(8)$  which is described in \cite[Section 5]{cartan1925principe}, \cite[Section 1]{mikosz2015triality}. To show that this outer automorphism does not induce an automorphism of $\SO(8)$ we follow the proof strategy from \cite{stackexchangeso8}. By \cite[Lemma 4]{varadarajan2001spin}, $\SO(8)$ has a subgroup isomorphic to $\SO(7)$ that is unique up to conjugation. Let's denote this conjugation class by $\Sigma_0$. Again by \cite[Theorem 3]{varadarajan2001spin}, $\SO(8)$ has exactly two conjugation classes of subgroups that are isomorphic to $\text{Spin}(7)$. We denote those conjugation classes by $\Sigma_1, \Sigma_2$.
Note that $t$ does lift to an automorphism $\tilde T \colon \text{Spin}(8) \rightarrow \text{Spin}(8)$. Furthermore, by \cite[Theorem 5]{varadarajan2001spin} the $\Sigma_i$ lift to three distinct conjugacy classes in $\text{Spin}(8)$, denoted by $\tilde \Sigma_i$ and $\{ \id, \tilde T, \tilde T^2\}$ acts transitively on the $\tilde \Sigma_i$. Now, if $t$ were to induce an automorphism on $\SO(8)$ it would also need to act transitively on the $\Sigma_i$. However, $\text{Spin}(7)$ is simply connected and $\SO(7)$ is not so this can't be.\\
Now that we've calculated the outer automorphisms of $\SO(2n)$ note that 
\[
\operatorname{End}(\SO(2n))/\operatorname{Inn}(\SO(2n)) = \text{Out}(\SO(2n)) \cup \{\text{triv}\}.
\]
Furthermore, the Weyl group of $\SO(2n)$ has $2^{n-1}\cdot n!$ elements and the statement follows.
\begin{figure}
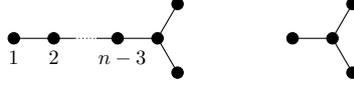

	\dynkin[scale = 1.5, labels = {1,2,,n-3}] D{} \hspace{1cm}
	\dynkin[scale = 1.5] D4
	\caption{Dynkin diagrams of $\mathfrak{so}(2n)$ and $\mathfrak{so}(8)$. Their construction is described in \cite[Theorem 21.11]{fulton2013representation}}.
\end{figure}
\end{proof}
The case for general Lie groups is more complicated. For example, there is no general classification of self maps of $BG$ if $G$ is a semi-simple Lie group.
However, we can understand them by looking at how such a map acts on the classifying space of its maximal torus:
\begin{theorem}[\cite{adams1976maps}, Theorem 1.1]\label{diagramtorus}
	Let $G$ be a compact connected Lie group with maximal torus $T$ and $f \colon BG \rightarrow BG$ be a continuous map. Then, there exists a group homomorphism $k \colon T \rightarrow T$ such that the following diagram commutes
\begin{equation}
	\begin{tikzcd}
	H^*(B\SO(2n); \Q) \ar[d, "i^*"] & H^*(B\SO(2n);\Q) \ar[l,"f^*"] \ar[d,"i^*"] \\
	H^*(BT;\Q) & H^*(BT;\Q) \ar[l,"Bk^*"]\mathrlap{.}
	\end{tikzcd}
\end{equation}
\end{theorem}
Picking such a $k$ for each $f$ defines a map (of sets)
\[
[BG,BG] \rightarrow \operatorname{End}(T).
\]
After fixing a choice of maximal torus, this map is well-defined up to an action of the Weyl group \cite[Corollary 1.8]{adams1976maps}.
If $G$ is a semi-simple Lie group, this map has been studied further in \cite{jackowski1995self}. Of special interest for us is the case where $BG \rightarrow BG$ is a rational equivalence:
\begin{theorem}[\cite{jackowski1995self}, Theorem 3.6]\label{theorem:injhomadmepi}
	Let $G$ be a compact connected Lie group with maximal torus $T$, Weyl group $W$ and lattice $\Lambda$. Then, there is an injective homomorphism of monoids
	\[
	\Theta \colon [BG,BG]_\Q \hookrightarrow (N_{\operatorname{Aut}(\Q \otimes \Lambda)}(W) \cap \operatorname{End}(\Lambda))/W
	\]
Here, $[BG, BG]_\Q$ denotes the set of homotopy classes of all maps $BG \rightarrow BG$ that are rational equivalences.
\end{theorem}
In other words for all rational equivalences $f \colon BG \rightarrow BG$, there exists a homomorphism $\Phi \colon T \rightarrow T$ such that $[\Phi] \in  (N_{\operatorname{Aut}(\Q \otimes \Lambda)}(W) \cap \operatorname{End}(\Lambda))/W$ and $g\vert_{BT} = B\Phi$.
\begin{remark}
	The elements of $(N_{\operatorname{Aut}(\Q \otimes \Lambda)}(W) \cap \operatorname{End}(\Lambda))/W$ are also called the \emph{admissible epimorphisms}. They have been named in \cite[p.5]{adams1976maps} since all maps $k$ making the diagram in \Cref{diagramtorus} commute are contained in this normalizer. \Cref{theorem:injhomadmepi} now tells us that rational equivalences are uniquely determined by their action on the maximal torus.\\
Note that other admissible epimorphisms could still induce maps on $BG$; however, those maps then will not be rational equivalences.
\end{remark}

\section{Proof of the main theorem}
Let $f \colon B\SO(d) \rightarrow B\SO(d)$ be a map covering the identity on $B\Aut(E_d)$. By restricting to arity 2 we obtain a commutative diagram
\[
\begin{tikzcd}
	B\SO(d) \ar[r, "f"] \ar[d] & B\SO(d)\ar[d]\\
	B\Aut(S^{d-1}) \ar[r, "\id"] & B\Aut(S^{d-1})\mathrlap{.}
\end{tikzcd}
\]
Therefore, to prove that all self maps of $E_d^{\SO(d)}$ are automorphisms it suffices to show:
\begin{theorem}\label{bigtheorem}
	Let $f\colon B\SO(d) \rightarrow B\SO(d)$ be a map covering the identity on the homotopy automorphisms $B\Aut(S^{d-1})$. Then, $f$ is homotopic to the identity and in particular an equivalence.
\end{theorem}
We will consider the odd and even case separately.

\subsection{Self Maps of $B\SO(2n+1)$}
Let $f\colon B\SO(2n+1) \rightarrow B\SO(2n+1)$ be a map such that the following diagram commutes up to homotopy
\begin{equation}\label{diagramwithi}
\begin{tikzcd}
	B\SO(2n+1) \ar[r, "f"] \ar[d, "i"] & B\SO(2n+1) \ar[d, "i"]\\
	B\text{Aut}(S^{2n}) \ar[r,"\id"] & B\text{Aut}(S^{2n})
\end{tikzcd}
\end{equation}
If $n =0$, then $\SO(1) = \{*\}$ and the theorem is trivially true.\\
If $n >0$, then by \cite[Lemma B.1]{krannich2021diffeomorphisms} there is a class $\varepsilon \in H^*(B\text{Aut}(S^{2n});\Q)$ such that
\begin{equation}\label{characteristicclassmaps}
	i^*(4 \epsilon)  = p_n \in H^{4n}(B\SO(2n+1);\Q).
\end{equation}
Therefore, by \eqref{diagramwithi} we have $f^*(p_n) = p_n$ in rational cohomology. By \Cref{cor:so2n+1maps}, $f \simeq \psi^k \circ B\alpha$ where 
\[
\alpha \colon \SO(2n+1) \rightarrow \SO(2n+1)
\]
is the identity or trivial map and $\psi^k$ is the unstable Adams operation of degree $k$. However, since $f^*(p_n) = p_n$ we obtain that $k = 1$ and $\alpha$ is the identity on $\SO(2n+1)$. Therefore, $f$ is homotopic to the identity.

\subsection{Self Maps of $B\SO(2n)$}
Let $f \colon B\SO(2n) \rightarrow B\SO(2n)$ be a map covering the identity on $B\text{Aut}(S^{2n-1})$.
Since $B\SO(2n)$ is simply connected, this map also covers the identity on the oriented homotopy automorphisms of the sphere $B\text{SAut}(S^{2n-1})$. The Euler class in $H^{2n}(B\text{SAut}(S^{2n-1});\Q)$ pulls back to the Euler class in $H^{2n}(B\SO(2n);\Q)$. As we cover the identity on $B\text{SAut}(S^{2n-1})$ we obtain
\[
f^*(e) = e.
\]
Thus, to prove our main theorem it suffices to show:
\begin{lemma}\label{lemmaeuler}
	Let $f\colon B\mathrm{SO}(2n) \rightarrow B \mathrm{SO}(2n)$ such that $f^*(e) = e$. Then, $f$ is homotopic to the identity.
\end{lemma}
\begin{proof}
For $n=1$ we have $B\SO(2) = \C P^\infty$, i.e. the cohomology ring $H^*(B\SO(2);\Z)$ is the polynomial ring generated by the Euler class. Therefore, if $f^*(e) = e$ we obtain that $f$ is the identity in cohomology. Therefore it is also homotopic to the identity by the classification of self maps of the circle.\\
Now, let $n = 2$. By \Cref{diagramtorus} there is a group homomorphism $k \colon T \rightarrow T$ making the diagram
\begin{equation}\label{diagramtorusrationalcohomology}
	\begin{tikzcd}
	H^*(B\SO(4); \Q) \ar[d, "i^*"] & H^*(B\SO(4);\Q) \ar[l,"f^*"] \ar[d,"i^*"] \\
	H^*(BT;\Q) & H^*(BT;\Q) \ar[l,"Bk^*"]
	\end{tikzcd}
\end{equation}
commute. By \Cref{exampletorus} $T$ has dimension 2. Therefore we obtain 
\[
H^*(BT;\Q) \cong \Q[x_1, x_2]
\]
with $|x_i|=2$. Furthermore, by \cite[Example 5]{may2005note} we have
\[
i^*(e) = x_1 x_2
\]
By the commuting diagram above, we obtain
\[
Bk^*(x_1 x_2) = x_1 x_2.
\]
Note that there are integers $a_{ij} \in \Z$ such that
\[
H^2(Bk;\Q) \cdot \begin{pmatrix}
	x_1\\
	x_2
\end{pmatrix} = \begin{pmatrix}
	a_{11}x_1 + a_{12}x_2\\
	a_{21}x_1 + a_{22}x_2
\end{pmatrix}.
\]
Since $Bk^*(x_1 x_2) = x_1 x_2$ we obtain
\[
(a_{11}x_1 + a_{12}x_2) (a_{21}x_1 + a_{22}x_2) = a_{11}a_{21}x_1^2 + a_{12}a_{22}x_2 + (a_{11}a_{22} + a_{12}a_{21})x_1x_2 = x_1 x_2.
\]
This implies that either $a_{11} = 0$ or $a_{12} = 0$.\\
If $a_{11} = 0$ we obtain that $a_{12}a_{21}= 1$, i.e. $a_{12} = a_{21} = \pm 1$.\\
Similarly, if $a_{12} = 0$ we obtain that $a_{11} = a_{22} = \pm 1$.\\
Either way, we obtain that $Bk^*$ is an equivalence in rational cohomology. As the vertical maps in \eqref{diagramtorusrationalcohomology} are injective by \cite[Example 5]{may2005note} we obtain that $f^*$ is an equivalence. Since $B\SO(4)$ is simply connected we obtain that $f$ is a rational equivalence.\\
Recall that by \Cref{theorem:injhomadmepi} we have an injection
\[
\Theta \colon [B\SO(4),B\SO(4)] \hookrightarrow (N_{\GL_2(\Q)}(W(\SO(4)) \cap \text{End}(\Z^2))/(W(\SO(4))
\]
which maps $f$ to $k$. By the above calculation of $Bk^*$ we already have that $k$ is an element of the Weyl group of $\SO(4)$. As $\Theta$ is an injective map, we obtain that $f$ is homotopy equivalent to the identity.\\
For $n > 2$ we use the classification of self maps of simple Lie groups. Recall that all maps $f$ arise from the composition of an unstable Adams operation with a map induced by an endomorphism of $\SO(2n)$.\\
However, note that $f$ cannot arise from a composition with an unstable Adams operation as $f^*(e) = e$. In addition, we know that all endomorphisms (up to inner automorphism) are either given by the trivial map, the identity map, or by conjugation with the reflection. The only one of those maps that maps the Euler class to itself is the identity map. Therefore we also obtain that $f$ is homotopic to the identity.
\end{proof}

\section{Variants}
\subsection{$\LO(n)$ acting on $E_n$}
We have a natural action of $\LO(n)$ on $\R^n$, i.e. we also have an operad $\smash{E_n}^{\LO(n)}$. Using the main theorem we can show:
\begin{theorem}
	All  self maps of $E_n^{\LO(n)}$ are invertible.
\end{theorem}
\begin{proof}
Let $f \colon B\LO(n) \rightarrow B\LO(n)$ be a map such that 
\[
\begin{tikzcd}
	B\LO(n) \ar[r, "f"] \ar[d] & B\LO(n) \ar[d]\\
	B\text{Aut}(S^{n-1}) \ar[r, "\id"] & B\text{Aut}(S^{n-1})
\end{tikzcd}
\]
commutes up to homotopy. Note that $f$ lifts to a map $\tilde f \colon B\SO(n) \rightarrow B\SO(n)$. As we have proven in the previous section, this map is an equivalence. Therefore, we know that $f$ induces an isomorphism on all homotopy groups except for the fundamental group. On the fundamental group we know that the vertical map is exactly the map
\[
\pi_0(\LO(n)) \rightarrow \pi_0(\text{Aut}(S^{n-1})).
\]
This map is an isomorphism and therefore $f$ is also an equivalence on the fundamental group.
Overall we obtain that $f$ is a homotopy equivalence.
\end{proof}

\subsection{$\SO(2n)$ acting on $E_{2n+1}$}
Through the inclusion $\SO(2n) \hookrightarrow \SO(2n+1)$ we obtain an action of $\SO(2n)$ on $E_{2n+1}$, that is we obtain an operad $\smash{E_{2n+1}}^{\SO(2n)}$.
\begin{theorem}
All self maps of $E_{2n+1}^{\SO(2n)}$ are invertible for $n\geq 1$.	
\end{theorem}
\begin{proof}
Let $f \colon B\SO(2n) \rightarrow B\SO(2n)$ be a map such that
\[
\begin{tikzcd}
	B\SO(2n) \ar[r, "f"] \ar[d] & B\SO(2n) \ar[d]\\
	B\text{Aut}(S^{2n}) \ar[r, "\id"] & B\text{Aut}(S^{2n})
\end{tikzcd}
\]
commutes up to homotopy.
The vertical maps factor through  the canonical map $j \colon B\SO(2n) \rightarrow B\SO(2n+1)$. For the $n$-th Pontryagin class $p_n \in H^{4n}(B\SO(2n+1);\Q)$ we obtain $j^*(p_n) =e^2 \in H^{4n}(B\SO(2n);\Q)$.\\
Furthermore, as stated in \eqref{characteristicclassmaps} there is a characteristic class $\epsilon\in H^*(B\text{Aut}(S^{2n};\Q))$ such that $4 \epsilon$ pulls back to $p_n$.\\
Putting this all together, we get that $f^*(e^2) = e^2$. By \cite[Theorem 15.9]{milnor1974characteristic} $H^*(B\SO(2n); \Q) \cong \Q[p_1, \dots, p_n]$ where $p_i$ denotes the $i$-th Pontryagin class, i.e. this ring is a division algebra. Therefore, $f^*(e) = \pm e$. If $f^*(e) = e$, we can directly apply \Cref{lemmaeuler} to obtain that $f$ is an equivalence. Otherwise, we can precompose $f$ with the map induced by conjugation by reflection (which is an equivalence) and then apply \Cref{lemmaeuler} to that new map to obtain the claim.
\end{proof}

\begin{remark}
We also have inclusions $\SO(k) \hookrightarrow \SO(2n+1)$ for $k < 2n$, that is we also have the operad $\smash{E_{2n+1}}^{\SO(k)}$. It is whether all endomorphisms of that operad are invertible.
\end{remark}


\subsection{Endomorphisms of the (framed) Swiss Cheese Operad}
The swiss cheese operad was first defined in \cite[Section 1]{voronov1999swiss}. We will recall the version defined in \cite[11]{kontsevich1999operads}:
\begin{definition}
Let $D^d\subset \R^{d}$ be the standard $d$-disk. Let $HD^d$ be the intersection of $D^d$ with the upper half plane, that is
\[
HD^d = \left \{(x_1, \dots,x_d) \in \R^d \mid x_1 \geq 0 \text{ and } \sum_{i=1}^d x_i^2 < 1\right \}.
\]
The \emph{d-dimensional swiss cheese operad} $\SC_d$ is a 2-colored operad with set of colors $S = \{D^d, HD^d\}$. The spaces of operation are given by:
\begin{itemize}
	\item If the output color is $D^d$ we have
\[
\SC_d(s_1, \dots, s_n; D^d) = \begin{cases}
	E_d(n) & \text{if } s_1 = \dots = s_n = D^d,\\
	\emptyset & \text{else.}
\end{cases}
\]
\item If the output color is $HD^d$ we have that
\[
\SC_d(s_1, \dots, s_n; HD^d) = \operatorname{emb}\left(\coprod_n s_i, HD^d\right)
\]
such that
\[
s_i = HD^d \Rightarrow j(\partial s_i) \subset \partial HD^d \quad \forall j \in \SC_d(s_1, \dots, s_n; HD^d)
\]
and can only scale by a positive number and translate the disks.
\end{itemize}
The operadic composition of the swiss cheese operad arises in the same way as the operadic composition of the little disk operad.
\end{definition}

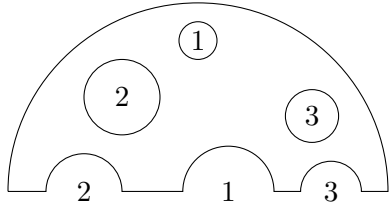
\begin{figure}[t]
\centering 
\begin{tikzpicture}[scale = 0.5]

\draw (-5,0) -- (-4,0)  arc(180:0:1)-- (-0.4,0) arc(180:0:1.2)-- (2.7,0) arc(180:0:0.8) -- (5,0) arc(0:180:5) ;
\draw (0,4) node{1} circle[radius = 0.5] ;
\draw (3,2) node{3} circle[radius = 0.7] ;
\draw (-2,2.5) node{2} circle[radius = 1] ;
\draw (-3,0) node{2};
\draw (0.8,0) node{1};
\draw (3.5,0) node{3};
\end{tikzpicture}
\caption{An element in $\SC(3,3)$}
\end{figure}

We will now define a framed version of the swiss cheese operad, following the approach from \cite{salvatore2003framed}:
\begin{definition}
The \emph{framed swiss cheese Operad} $\SC_d^{\SO(d)}$ is a 2-colored operad with set of colors $S = \{D^d,HD^d\}$. Its set of operations is given by
\[
\SC_d^{\SO(d)}(s_1, \dots, s_n; D^d) = \begin{cases}
	E_d^{\SO(d)}(n) & \text{if } s_1 = \dots s_n = D^d, \\
	\emptyset & \text{else,}
\end{cases}
\]
if the output color is $D^d$ and
\[
\SC_d^{\SO(d)}(s_1, \dots, s_n; HD^d) = \SC_d(s_1, \dots, s_n;HD^d) \times \SO(d)^{\times k} \times \SO(d-1)^{\times n-k}
\]
if the output color is $HD^d$. Here, $k$ is the number of $s_i$ with $s_i = D^d$. The operadic composition arises in the same way as the operadic composition of the framed little disk operad.
\end{definition}
An element in $\smash{\SC_d}^{\SO(d)}(s_1, \dots, s_n; HD^d)$ should be interpreted as an element in $\SC_d(s_1, \dots, s_n;HD^d)$ where we assign each embedded full disk a rotation and each embedded half disk an orientation of its boundary.
\subsubsection{Endomorphisms of the swiss cheese operad preserving colors}
\begin{theorem}
\begin{enumerate}
\item Let $f \colon \SC_d \rightarrow \SC_d$ be an endomorphism of the swiss cheese operad such that $f$ is the identity on the colors. Then, $f$ is an automorphism.\\
\item Let $g \colon \SC_d^{\SO(d)} \rightarrow \SC_d^{\SO(d)}$ be an endomorphism of the framed swiss cheese operad such that it is the identity on the colors. Then, $g$ is an automorphism.
\end{enumerate}
\end{theorem}
\begin{proof}
We will first show that $f$ is an equivalence on each space of operations. This will then conclude the first part of the theorem.\\
Note that we have an inclusion $E_d \hookrightarrow \SC_d$ as we have that $E_d(n) = \SC_d(\amalg_n D^d;D^d)$. Similarly, we have an inclusion $E_{d-1} \hookrightarrow \SC_d$ since $E_{d-1}(n) \simeq \SC_d(\amalg_n HD^d; HD^d)$. As we assumed $f$ to be color preserving, we obtain that $f$ restricts to an endomorphism of both $E_d$ and $E_{d-1}$. By \cite[Theorem B]{horel2022two} we obtain that both of those restrictions are automorphisms.\\
It remains to show that $f$ is invertible for general inputs and output $HD^d$. We have a map induced by operadic composition with the 0-ary element:
\begin{equation}\label{mapconfig}
\SC_d(s_1, \dots, s_n; HD^d) \rightarrow  \SC_d(\amalg_{n-k}HD^d; HD^d) \times \SC_d(\amalg_k D^d; HD^d)
\end{equation}
where $k \leq n$ is the number of $s_i$ such that $s_i = D^d$. We will show that this map is an equivalence. By taking the middle point of each (half) disk, we obtain a map
\[
\SC_d(s_1, \dots, s_n; HD^d) \rightarrow \operatorname{Conf}_n(HD^d).
\]
Note that the image of this map is by definition contained in $\operatorname{Conf}_{n-k}(\partial HD^d) \times \operatorname{Conf}_k(\operatorname{Int}(HD^d))$, i.e. we obtain a map
\[
\SC_d(s_1, \dots, s_n; HD^d) \rightarrow \operatorname{Conf}_{n-k}(\partial HD^d) \times \operatorname{Conf}_k(\operatorname{Int}(HD^d)).
\]
This map is an equivalence; its fiber only consists of the possible radii of the disks which is a contractible space.\\
The right hand side in \eqref{mapconfig} also maps into $\operatorname{Conf}_{n-k}(\partial HD^d) \times \operatorname{Conf}_k(\operatorname{Int}(HD^d))$ by taking the mid points of all the (half) disks. For the same reason as above, this map is an equivalence. Therefore, the map \eqref{mapconfig} fits into the following commutative diagram
\[
\begin{tikzcd}
	\SC_d(s_1, \dots, s_n; HD^d)\ar[r] \ar[dr, "\simeq"] & \SC_d(\amalg_{n-k}HD^d; HD^d) \times \SC_d(\amalg_k D^d; HD^d) \ar[d,"\simeq"] \\
	&\operatorname{Conf}_{n-k}(\partial HD^d) \times \operatorname{Conf}_k(\operatorname{Int}(HD^d))
\end{tikzcd}
\]
and is an equivalence.\\
Furthermore, we have $\operatorname{Conf}_{n-k}(\partial HD^d) \times \operatorname{Conf}_k(\operatorname{Int}(HD^d)) \simeq E_{d-1}(n-k) \times E_d(k)$.
Therefore, we obtain a commutative diagram
\[
\begin{tikzcd}
	\SC_d(s_1, \dots, s_n; HD^d) \ar[r] \ar[d, "f"] & E_{d-1}(n-k) \times E_d(k) \ar[d, "\simeq"]\\
	\SC_d(s_1, \dots, s_n; HD^d) \ar[r] & E_{d-1}(n-k) \times E_d(k)
\end{tikzcd}
\]
We obtain that the right vertical map is an equivalence and thus so is $f$.\\
To show that $g$ is an equivalence on each space of operations we consider the same restriction maps as above. $g$ then becomes a map of the framed little disk operad on each space of operations. By the main theorem of this paper, all of those maps must be equivalences and thus $g$ is as well.
\end{proof}

\begin{remark}
The first part of this theorem also follows from the main theorem in \cite{turchin2024mapping}. However, the author uses a Fulton-Macpherson version of the swiss cheese operad to prove his theorem; here we only rely on \cite[Theorem B]{horel2022two} to prove the statement.
The second part of this theorem relies on the main theorem of this paper and does not follow from \cite{turchin2024mapping}.
\end{remark}

\subsubsection{Endomorphism of the swiss cheese operad not preserving colors}
\begin{theorem}
\begin{enumerate}
	\item There are no endomorphisms of the swiss cheese operad mapping the color $D^d$ to $HD^d$.
	\item In contrast, we can construct an endomorphism of the swiss cheese operad such that both colors map to $D^d$; that is we can construct an endomorphism of $\SC_d$ factoring through $E_d$.
\end{enumerate}
\end{theorem}
\begin{proof}
	Let $f \colon \SC_d \rightarrow \SC_d$ be an endomorphism such that the color $D^d$ gets mapped to $HD^d$.\\
	If $f$ were to map the color $HD^d$ to $D^d$, then $f$ would induce a map
	\[
	f \colon \SC_d(D^d; HD^d) \rightarrow \SC_d(HD^d; D^d) \simeq \emptyset
	\]
	which is not possible.\\
	Now, if $f$ were to map the color $HD^d$ to $HD^d$ we would obtain maps
	\[
	E_{d}(n) \simeq \SC_d(D^d, \dots, D^d; D^d) \xrightarrow{f} \SC_d(HD^d, \dots, HD^d) \simeq E_{d-1}(n).
	\]
	This induces a map of operads $E_d \rightarrow E_{d-1}$. Furthermore, we have a map of operads $E_{d-1} \rightarrow E_d$ given by inclusion. The composition of those two maps is then a self map of $E_{d-1}$, i.e. it must be an equivalence. In particular, it will be an equivalence on 2-ary operations. However, on 2-ary operations this map factors as
	\[
	S^{d-1} \hookrightarrow S^d \rightarrow S^{d-1}
	\]
which is a nullhomotopic map. Therefore, this map is not an equivalence and $f$ cannot exist.\\
To prove the second part of this theorem we will construct an endomorphism $g$ of the swiss cheese operad such that $HD^d$ and $D^d$ both map to $D^d$: 
We first define
\[
\SC_d(s_1, \dots, s_n; HD^d) \rightarrow \SC_d(D^d, \dots, D^d; D^d).
\]
 Let $\langle j_1, \dots, j_n \rangle \in \SC_d(s_1, \dots, s_n; HD^d)=\operatorname{emb}(s_1 \amalg \dots \amalg s_n, HD^d)$. If $s_i = D^d$, we will not change $j_i$. If $s_i = HD^d$, we will mirror what $j_i$ is doing onto the lower half.
$g$ will be the identity if the output color is $D^d$.\\
$g$ is now a non-trivial and non-invertible map $\SC_d \rightarrow \SC_d$.
\end{proof}
\begin{remark}
	Both the theorem and the remark also apply to the framed swiss cheese operad using the main theorem from this paper.
\end{remark}

\printbibliography
\end{document}

%% file: preamble.tex
\usepackage[dvipsnames,svgnames,x11names,hyperref]{xcolor}
\usepackage{dsfont,textgreek,url,graphicx,verbatim,amssymb,enumerate,stmaryrd,booktabs,lmodern,mathtools,microtype,mathabx}
\usepackage[colorlinks,citecolor=Mahogany,linkcolor=Mahogany,urlcolor=Mahogany,filecolor=Mahogany]{hyperref}
\usepackage[capitalize]{cleveref}
\usepackage[mathscr]{euscript}
\usepackage[margin=1.5in]{geometry}
\usepackage{tikz,tikz-cd}
\usetikzlibrary{matrix,calc,positioning,arrows,decorations.pathreplacing,patterns,arrows}

\newtheorem{theorem}{Theorem}[section]
\newtheorem{lemma}[theorem]{Lemma}
\newtheorem{proposition}[theorem]{Proposition}
\newtheorem{corollary}[theorem]{Corollary}
\newtheorem*{corollary*}{Corollary}

\newtheorem{atheorem}{Theorem}

\theoremstyle{definition}
\newtheorem{definition}[theorem]{Definition}

\theoremstyle{remark}
\newtheorem{example}[theorem]{Example}

\newtheorem{remark}[theorem]{Remark}

\renewcommand{\rm}[1]{{\mathrm{#1}}}

\newcommand{\op}{\rm{Op}_\infty}

\DeclareFontFamily{U}{min}{}
\DeclareFontShape{U}{min}{m}{n}{<-> udmj30}{}

\AtBeginDocument{%
	\def\MR#1{}
}